\documentclass[a4paper, reqno, 12pt]{amsart}

\newif\iffinal

\pdfoutput=1

\usepackage[utf8x]{inputenc}
\usepackage[T1]{fontenc}
\usepackage{lmodern}
\usepackage[english]{babel}

\usepackage{amssymb}
\usepackage{colonequals}
\usepackage{mathtools}
\usepackage{hyperref}
\usepackage{colortbl}
\usepackage{color}
\usepackage{stmaryrd}
\usepackage{multirow}
\usepackage[left=2.5cm,top=3.1cm,right=2.5cm,bottom=3.1cm,bindingoffset=0.7cm]{geometry}

\DeclarePairedDelimiter{\norm}{\lVert}{\rVert}

\newtheorem{theorem}{Theorem}[section]

\newtheorem{corollary}[theorem]{Corollary}

\newtheorem{lemma}[theorem]{Lemma}
\theoremstyle{definition}
\newtheorem{definition}[theorem]{Definition}
\newtheorem{notation}[theorem]{Notation}

\newtheorem{remark}[theorem]{Remark}

\DeclareMathOperator{\Int}{Int}
\DeclareMathOperator{\Spec}{Spec}
\newcommand{\fixdiv}{\mathsf{d}}
\newcommand{\N}{\mathbb{N}}
\newcommand{\Z}{\mathbb{Z}}
\newcommand{\R}{\mathcal{R}}
\newcommand{\F}{\mathcal{F}}
\newcommand{\p}{\mathcal{P}}
\newcommand{\Q}{\mathcal{Q}}

\newcommand{\Mod}[1]{\ (\mathrm{mod}\ #1)}
\def\card#1{\left|#1\right|}

\makeatletter
\@namedef{subjclassname@2020}{\textup{2020} Mathematics Subject Classification}
\makeatother

\usepackage{letltxmacro}
\usepackage{todonotes}

\author[Z.~Kansiime]{Zaituni Kansiime}
\address{Department of Mathematics \\
 Kabale University\\
 Plot 364 Block 3 Kikungiri Hill
Kabale\\ Uganda}
\email{\href{mailto: kansiimezaituni2018@gmail.com}{kansiimezaituni2018@gmail.com}}

\author[S.~Luambano]{Sholastica Luambano}
\address{Department of Mathematics\\
 University of Dodoma\\ P.O. Box 259\\ Dodoma\\Tanzania}
\email{\href{mailto:sholawilgis@gmail.com}{sholawilgis@gmail.com}}

\author[S.~Nakato]{Sarah Nakato}
\address{Department of Mathematics \\
 Kabale University\\
 Plot 364 Block 3 Kikungiri Hill
Kabale\\ Uganda}
\email{\href{mailto: snakato@kab.ac.ug}{snakato@kab.ac.ug}}

\author[H.~Nalule]{Hadijah Nalule}
\address{Department of Mathematics \\
 Kabale University\\
 Plot 364 Block 3 Kikungiri Hill
Kabale\\ Uganda}
\email{\href{mailto: hadijahnalule4@gmail.com}{hadijahnalule4@gmail.com
}}

\author[Y.~Ndayikunda]{Yvette Ndayikunda}
\address{Department of Mathematics\\
 University of Burundi\\
   B.P. 2700, Bujumbura\\ Burundi}
\email{\href{mailto:yvettendayikunda18@gmail.com}{yvettendayikunda18@gmail.com}}

\title[{Sets of lengths of integer-valued polynomials}]{Sets of lengths
of integer-valued polynomials on prime ideals of principal ideal domains}


\begin{document}

\maketitle

\begin{abstract}
Let $D$ be a principal ideal domain with infinite spectrum such that
for every nonzero prime ideal $M$ of $D$, the residue field $D/M$ is
finite. Let $K$ be the quotient field of $D$. We investigate sets of
lengths in the ring of integer-valued polynomials on $M$,
$\Int(M, D) = \{f \in K[x] ~ \vert ~ f(M) \subseteq D\}$.
For every multiset of integers $1 < z_1 \leq z_2 \leq \cdots \leq z_n$,
we explicitly construct an element of $\Int(M, D)$ with exactly
$n$ essentially different factorizations into irreducible elements of
$\Int(M, D)$ whose lengths are $z_1, z_2, \ldots, z_n$.
Furthermore, we show that $\Int(M, D)$ is not a transfer Krull domain.
These results spark off the study of sets of lengths in the rings
$\Int(S, D) \neq \Int(D)$, where $S$ is an infinite subset of $D$.
\end{abstract}

\section{Introduction}
The study of non-unique factorizations in algebraic structures is a
fundamental area of research in algebra and number theory.
A key invariant in this field is the set of lengths of factorizations, which
records the possible numbers of irreducible factors in essentially
different factorizations of a given element. Sets of lengths
describe the factorization behaviour of the algebraic structure,
and from them, several other important factorization invariants can be
derived.
Understanding these sets offers important insights into the arithmetic
of the structure.

In half-factorial domains, sets of lengths are singletons, but in many
monoids and integral domains, factorizations are often highly non-unique.
Geroldinger~\cite{GeA2016:sets-of-L} highlights several
studies on sets of lengths in monoids. In this work, we examine sets
of lengths in rings of integer-valued polynomials on subsets.

Let $D$ be an integral domain with quotient field $K$, and let $S$ be a
subset of $D$. The ring of integer-valued polynomials on $S$ is defined as
\[\Int(S, D) = \{f \in K[x] ~\vert ~ f(S) \subseteq D\}.\]
When $S = D$, we get the classical ring of
integer-valued polynomials, $\Int(D) = \{f \in K[x]~ \vert ~ f(D) \subseteq D\}$.
There are proper subsets $S \subset D$ where $\Int(S, D) = \Int(D)$,
see~\cite{CahenChabert:2016:integer-valued-polynomials-update, gilmer1989sets}.
In this paper, we focus on the case where $\Int(S, D) \neq \Int(D)$.
In this case, a polynomial $f \in \Int(D)$ may exhibit different
factorization behavior in $\Int(S, D)$. For example, the polynomial $x$
is irreducible in $\Int(D)$ but does not admit a factorization into irreducibles in
$\Int(\{0\}, D)$, since for each nonzero non-unit $a \in D$, we have
$$x = a \cdot \frac{x}{a}$$ and the factor $\frac{x}{a}$ further decomposes without
stabilizing into irreducibles. For detailed investigations on atomicity and irreducibility
in $\Int(S, D)$, see,
\cite{AD-CJ-SC-SW1995some, CahenChabert:1997:IVP, chapman2005irreducible}.

Following the notation in \cite{VF-DW2024:S-O-L-Krull}, we say
that $\Int(S, D)$ has full system of
multisets of lengths if for all positive integers
$1 < z_1 \leq z_2 \leq \cdots \leq z_n$, there exists
a polynomial $H \in \Int(S, D)$ with exactly
$n$ essentially different factorizations into irreducible elements of
$\Int(S, D)$ whose lengths are $z_1, z_2, \ldots, z_n$.
Let $D$ be a principal ideal domain with infinite spectrum such that
for every nonzero prime ideal $M$ of $D$, the residue field $D/M$ is
finite. In Theorem~\ref{Thm:prescribed-lengths}, we show that $\Int(M, D)$
has full system of multisets of lengths. This finding sparks off
a more general study of sets of lengths in rings $\Int(S, D) \neq \Int(D)$,
where $S$ is an infinite subset of $D$.

Note that significant progress has been made on sets of lengths in
$\Int(D)$. Specifically, Frisch~\cite{SF2013:Sets-of-lengths}
examined the case $D = \mathbb{Z}$. Subsequently, Frisch, Rissner, and
the third author~\cite{SF-SN-RR2019:Sets-of-lengths} extended this analysis
to Dedekind domains with infinitely many maximal ideals, all of finite index.
More recently, Fadinger, Frisch, and Windisch~\cite{VF-SF-DV2023:S-O-L-DVR}
investigated the case when $D$ is a valuation ring of a global field, while
Fadinger and Windisch~\cite{VF-DW2024:S-O-L-Krull} focused on Krull domains $D$
that admit a prime element with finite residue field. In all these contexts,
the authors established that $\Int(D)$ has full system of multisets of lengths.

However, the results established for $\Int(D)$ do not extend directly to
$\Int(S, D)$, even in the canonical case where $D = \mathbb{Z}$.
Currently, no characterization exists for subsets $S \subsetneq \mathbb{Z}$
such that $\Int(S, \mathbb{Z})$ has full system of (multi)sets of lengths.
Our results contribute to closing this gap by offering new insights that advance
the development of such a characterization.

Similar results about full systems of sets of lengths have been obtained using
transfer homomorphisms to block monoids; see, for instance,~\cite{FK1999factorization}.
Frisch~\cite{SF2013:Sets-of-lengths} showed that there is no transfer homomorphism
from the multiplicative monoid of $\Int(\mathbb{Z})$ to a block monoid.
Consequently, $\Int(\mathbb{Z})$ is not a transfer Krull domain.
Later, Frisch, Rissner, and the third author~\cite{SF-SN-RR2019:Sets-of-lengths}
established that $\Int(D)$ is not a transfer Krull domain whenever $D$ is
a Dedekind domain with infinitely many maximal ideals, all of finite index.
In Theorem~\ref{Thm:non-transfer-krull}, we show that  $\Int(M, D)$
is not a transfer Krull domain. This adds $\Int(M, D)$ to the small
list of naturally occurring domains that are not transfer Krull domains.

\section{Preliminaries}
\subsection{Factorization terms}
We provide a brief overview of the essential factorization terms for
this study. For a more detailed introduction to non-unique factorizations,
we refer the reader to the monograph of Geroldinger and
Halter-Koch~\cite{GeHa2006:NonUnifacts}.

Let $D$ be an integral domain, and let
$e, a, b \in D$ be nonzero elements.
We say that $e$ is a \emph{unit} in $D$ if there exists a nonzero $c \in D$
such that $ec = 1.$ For instance, the rings $D, D[x], \Int(D),$ and
$\Int(S, D)$ all have the same units.

We call $r$ an \emph{irreducible} element of $D$ if it cannot be written as the
product of two non-units of $D$. For example, primitive linear polynomials are
irreducible in $D[x]$ and $\Int(D)$, but not all remain irreducible
in $\Int(S, D) \neq \Int(D)$,~cf.~Remark~\ref{remark:linear}.

A \emph{factorization} of $a \in D$ is an expression of $a$ as a product of irreducible
elements of $D$. That is, \[a = a_{1} \cdot a_{2} \cdots \cdot a_{n},\]
where $n \geq 1$ and $a_i$ is irreducible in
$D$ for all $1 \leq i \leq n$.
The \emph{length of a factorization} of $r \in D$ is the number of irreducible factors within
the factorization of $r$.

We say that $a, b$ are \emph{associated} in $D$ (written $a \sim b$) if there exists
a unit $e \in D$ such that $a = eb$.
Two factorizations of the same element
$a = a_{1} \cdot a_2 \cdots a_{n} = r_{1} \cdot r_2 \cdots r_{m}$ are said to be
\emph{essentially the same} if $n = m$ and, after a suitable re-indexing, $a_{j} \sim r_{j}$
for $1 \leq j \leq m$. Otherwise, the factorizations  are called \emph{essentially different}.
For example, in $D = \mathbb{Z}[\sqrt{-14}]$, \[81 = 3 \cdot 3 \cdot 3 \cdot 3 = (5 + 2\sqrt{-14})(5 - 2\sqrt{-14})\]
are essentially different factorizations of 81 in $D$. In this case, 81 has two essentially
different factorizations, one of length four and the other of length two. This factorization
behaviour naturally led to the widely studied notion of sets of lengths, defined as follows.

\begin{definition}
Let $D$ be an integral domain. The \emph{set of lengths} of $r \in D$ is the set of
all natural numbers $n$ such that $r$ has a factorization of length $n$.
We denote this by $L(r)$.
\end{definition}
For instance, in
$\mathbb{Z}[\sqrt{-14}],$ every factorization of $81$ is essentially
the same as either $81 = 3 \cdot 3 \cdot 3 \cdot 3$ or
$81 = (5 + 2\sqrt{-14})(5 - 2\sqrt{-14})$. Similarly, 15 factors as
$15 = 3 \cdot 5 = (1 + \sqrt{-14})(1 - \sqrt{-14})$.
Therefore $L(81) = \{2, 4\}$ and $L(15) = \{2, 2\} = \{2\} $. In the latter case,
$\{2, 2\}$ is regarded as a multiset of lengths of 15.

Following the notation in \cite{VF-DW2024:S-O-L-Krull}, we say
that $D$ has \emph{full system of multisets of lengths} if for all positive integers
$1 < z_1 \leq z_2 \leq \cdots \leq z_n$, there exists
an element $r \in D$ with exactly $n$ essentially different factorizations
into irreducible elements of $D$ whose lengths are $z_1, z_2, \ldots, z_n$.

\subsection{Transfer mechanisms} Some results on sets of lengths were
obtained using transfer homomorphisms. These homomorphisms transfer
factorization properties from a simpler suitable monoid to the monoid
under investigation. For an integral domain $D$, the monoid of interest
is $(D \setminus \{0\}, \cdot).$
\begin{definition}\cite[Definition 3.2.1]{GeHa2006:NonUnifacts}
Let $H$ and $M$ be commutative monoids, and denote by $H^{\times}$ and $M^{\times}$
their respective groups of units. A homomorphism
$\theta\colon H \longrightarrow M$ is called a \emph{transfer homomorphism}
if it has the following properties:
\begin{enumerate}
\item $M = \theta(H)M^{\times}$ and $\theta^{-1}(M^{\times})=  H^{\times}$.
\item If $u \in H$ and $b$,~$c \in M$ such that $\theta(u) = bc$,
 then there exists $v$,~$w \in H$
such that $u=vw$ and $\theta(v) \sim b$ and $\theta(w) \sim c$.
\end{enumerate}
\end{definition}
For example, if $D$ is a half-factorial domain, then the map $$\theta: (D \setminus \{0\}, \cdot)  \longrightarrow (\N_0, +), \quad
  \theta(r) = l(r) \text{ for all } r \in D \setminus \{0\},$$ is a transfer homomorphism, where $l(r)$ denotes the length of a factorization
  of $r$.

The suitable monoids that are widely used in transfer mechanisms are the
monoids of zero-sum sequences defined in the following.
\begin{definition}\label{block}Let $G$ be an additively written Abelian
group and $G_{0} \subseteq G$ a nonempty subset. Let $\mathcal{F}(G_{0})$
denote the free Abelian monoid with basis $G_{0}.$
\begin{enumerate}
\item The elements of $\mathcal{F}(G_{0})$ are called \emph{sequences }over $G_0$ and are of the form
\[S = \prod_{g \in G_{0}}g^{n_g}\]
where $n_g = v_g(S) \in \N \cup \{0\}, n_g = 0$ for almost all $g \in G_0$.
\item The \emph{length} of a sequence $S$ is
\[\card{S} = \sum_{g \in G_{0}}v_{g}(S) ~~\in ~\N \cup \{0\}\]
and the \emph{sum }of $S$ is
\[\sigma(S)= \sum_{g \in G_{0}}v_{g}(S)g ~~\in ~G.\]
\item The monoid
\[\mathcal{B}(G_{0}) = \Biggl\{S \in \mathcal{F}(G_{0})~ \bigl\vert~ \sigma(S) =0 \Biggr\}\]
is called the \emph{block monoid} over $G_{0}$ or the \emph{monoid of zero-sum sequences}.
\item An element $S \in \mathcal{B}(G_{0})$ is irreducible in $\mathcal{B}(G_{0})$ if no non-empty proper subsum equals 0.
\end{enumerate}
\end{definition}
\begin{lemma}\cite[Lemma 6.4.4.]{GeHa2006:NonUnifacts}\label{lemma:blockMonoid}
Let $G$ be an additively written Abelian group and $G^{\bullet} = G \setminus \{0\}$.
Let $U, V$ be irreducible elements of $\mathcal{B}(G^{\bullet})$.
\begin{enumerate}
\item max$L(UV) \leq$ min$\{\card{U}, \card{V}\}$.
\item max$L(UV) =$ max$\{\card{U}, \card{V}\}$ if and only if $V = -U$.
\item $\card{U} =$ max$\{\text{max}L(UU^{\prime}) \mid U^{\prime}
\text{ is irreducible in }\mathcal{B}(G)\}$.
\end{enumerate}
\end{lemma}

\begin{definition}\cite[Section 4]{GeA2016:sets-of-L}
Let $G_0$ be a subset of an Abelian group $G$. A commutative monoid
$H$ is said to be a \emph{transfer Krull monoid} (over $G_0$) if
there exists a transfer homomorphism $\theta : H \longrightarrow \mathcal{B}(G_{0})$.
An integral domain $D$ is a \emph{transfer Krull Domain} if $D \setminus \{0\}$ is a \emph{transfer Krull monoid}.
\end{definition}
For example, every commutative Krull monoid is a transfer Krull monoid.
Therefore Krull domains are transfer Krull domains,
cf.~\cite[Proposition 4.3]{GeA2016:sets-of-L}.

\begin{remark}\label{Remark:notTK}
Let $\theta\colon H \longrightarrow M$ be a transfer homomorphism.
An element $r \in H$ is irreducible in $H$ if and only if $\theta(r)$
is irreducible in $M$ (see~\cite[Proposition 3.2.3]{GeHa2006:NonUnifacts}).
Furthermore, Lemma~\ref{lemma:blockMonoid} shows that in a block monoid,
the lengths of factorizations of elements of the form $UV$, where $U$ and
$V$ are irreducible and $U$ is fixed, are bounded by a constant that depends
only on $U$. Therefore, these results together imply that every transfer Krull
monoid possesses this property.
\end{remark}

\subsection{Integer-valued polynomials on subsets}
We provide an overview of the concepts on integer-valued
polynomials on subsets and their irreducibility needed for this
paper. We refer the reader
to the monograph of Cahen and Chabert~\cite{CahenChabert:1997:IVP} and
articles~\cite{CahenChabert:2016:integer-valued-polynomials-update, chapman2005irreducible}
for a detailed account.
\begin{definition}
Let $D$ be an integral domain and $S \subseteq D$.
The \emph{ring of integer-valued polynomials on} $S$ is defined as
\[\Int(S, D) = \{f \in K[x] ~\vert ~ f(S) \subseteq D\}.\]
\end{definition}
The ring $\Int(S, D)$ contains $\Int(D) = \{f \in K[x]~ \vert ~ f(D) \subseteq D\}$.
There are proper subsets $S$ of $D$ for which $\Int(S, D) = \Int(D)$.
Furthermore, there are subsets $E, F \subset D$ such that
$\Int(E, D) = \Int(F, D)$. For example, $\Int(\N, \Z) = \Int(\Z)$, and
if $\p \subset \Z$ is the set of prime numbers and
$E = \p \cup \{-1, 1\}$, then $\Int(\p, \Z) = \Int(E, \Z) \neq \Int(\Z)$,~
cf.~\cite{CahenChabert:2016:integer-valued-polynomials-update, gilmer1989sets}.

In this work, we focus on cases where $\Int(S, D) \neq \Int(D)$ and
$\Int(E, D) \neq \Int(F, D)$ for $E \neq F$.
In particular, let $D$ be a principal ideal domain with infinite spectrum
such that for every nonzero prime ideal $M$ of $D$, the residue field $D/M$ is
finite. Then $\Int(M, D)$ strictly contains $\Int(D)$.
For instance, if $p$ denotes the generator of $M$ and $n \in \N$, then the
polynomial $f = \frac{x^n + \sum_{i=0}^{n-1}px^i}{p}$ is an element of
$\Int(M, D) \setminus \Int(D)$, but $f \not\in \Int((q), D)$ for primes $q \neq p$.

\begin{definition}
Let $D$ be a principal ideal domain, $S \subseteq D$, and $f \in \Int(S, D)$.
The \emph{fixed divisor} of $f$ over $S$, denoted $\fixdiv(S, f)$, is defined as
\[ \fixdiv(S, f) = \gcd\{f(s)~\vert~s \in S \}. \]
The polynomial $f$ is called \emph{image-primitive} if $\fixdiv(S, f) = 1$.
\end{definition}
Let $g \in D[x]$ be non-constant and $b \in D$ nonzero. Then the polynomial $\frac{g}{b} \in K[x]$ is an element
of  $\Int(S, D)$ if and only if $b$ divides $\fixdiv(S, g)$.

We need to track the possible prime divisors of $\fixdiv(S, g)$.
For that purpose, we fix the following notation
that helps us to capture the residue classes of $D/pD$
represented in $S \subset D$.

\begin{notation}
Let $D$ be a principal ideal domain and $p \in D$ a prime element such
that $\card{D/pD} < \infty$. Let $S \subseteq D$ be
non-empty. Set
\begin{equation*}
\R_S(p) = \{ \bar{a} \in D/pD \mid \exists~s \in S \text{ so that } a \equiv s \Mod{p}\}.
\end{equation*}
Then $\R_S(p)$ is called \emph{a set of residues of $p$ with
respect to $S$}. If $\card{\R_S(p)} = t$, then
a set $\{\bar{b_1}, \bar{b_2}, \ldots, \bar{b_t}\}$ is called a \emph{complete set of residues}
 of $p$ with respect to $S$ if
  \begin{enumerate}
  \item for all $1 \leq i \leq t$, there exists $\bar{a} \in \R_S(p)$ such
   that $b_i \equiv a \Mod{p}$, and
  \item $b_i \not\equiv b_j \Mod{p}$ for $i \neq j$.
 \end{enumerate}
\end{notation}
For example, if $S = p\Z$, then $\R_S(p) = \{\bar{0} \}$ and $\R_S(q) = \Z/q\Z$ for
prime numbers $q \neq p$.

\begin{remark}\label{remark:linear}
Let $D$ be a principal ideal domain and $S \subseteq D$.
\begin{enumerate}
\item If $g \in D[x]$ and $M = (p)$ is a prime ideal of $D$ such that
$p$ divides $\fixdiv(S, g)$, then $g \in M[x]$ or $\R_S(p) \leq \deg(g)$.
\item If a prime element $p \in D$ divides the fixed divisor of
$f \in \Int(S, D)$, then $f = p \cdot \frac{f}{p}$ is a factorization
of $f$ in $\Int(S, D)$ (not necessarily into irreducibles). It
follows that irreducible polynomials in $\Int(S, D)$ must be image-primitive.
\item If $f \in D[x]$ is irreducible in $\Int(D)$, then $f$ is irreducible
in $\Int(M, D)$ if and only if $f$ is image-primitive in $\Int(M, D)$.
This implies that even the basic polynomials like $x$ do not remain irreducible
in $\Int(M, D)$, which demonstrates the distinction between the factorization
behaviour of the two rings. More generally, Chapman and McClain characterised the
irreducible elements of $\Int(S, D)$ as follows.
\end{enumerate}
\end{remark}
\begin{theorem}\label{Thm-irr}\cite[Theorem 2.8]{chapman2005irreducible}
Let $D$ be a principal ideal domain and $S \subseteq D$. Let $f$
be a nonconstant primitive polynomial in $D[x]$. Then $\frac{f}{\fixdiv(S, f)}$
is irreducible in $\Int(S, D)$ if and only if either $f$ is irreducible in $D[x]$ or
for every pair of nonconstant polynomials $f_1, f_2 \in D[x]$ with
$f = f_1f_2$, $\fixdiv(S, f) \nmid \fixdiv(S, f_1) \times \fixdiv(S, f_2)$.
\end{theorem}
From now on, we write $\fixdiv(f)$ instead of $\fixdiv(S, f)$, with the underlying set
$S$ understood from context.
\section{Necessary Lemmas}
The following lemma guarantees that for each $n \in \N$, we have finitely many
maximal ideals with $\card{\R_M(p)} \leq n$. This finiteness condition
is crucial since we apply the Chinese Remainder Theorem in our constructions,
and it further ensures the existence of infinitely many maximal ideals with
$\card{\R_M(p)} > n$.

\begin{lemma}\label{Lemma:FMP}
Let $D$ be a principal ideal domain and $M$ a maximal ideal of $D$.
Then for each given $q \in \N$, there are at most finitely many maximal
ideals $(p)$ of $D$ with $\card{\R_M(p)} = q$.
\end{lemma}
\begin{proof}
$\card{\R_M(p)} = 1$ if and only if $M = (p)$. For $\card{\R_M(p)} \ge 2$,
the argument follows exactly as in the proof of
\cite[Proposition 2.1]{SF-SN-RR2019:Sets-of-lengths} with $\norm P$
replaced with $\card{\R_M(p)}$.
\end{proof}

The following lemma helps us to replace families of polynomials of
$D[x]$ with corresponding irreducibles of $K[x]$ without altering valuations.
We omit the proof, as it is identical to that of~\cite[Lemma 6.2]{hiebler2023characterizing},
a variation of~\cite[Lemma 3.3]{SF-SN-RR2019:Sets-of-lengths}.

\begin{lemma}\label{lemma:replacements}\cite[Lemma 3.3]{SF-SN-RR2019:Sets-of-lengths}, \cite[Lemma 6.2]{hiebler2023characterizing}
  Let $D$ be a principal ideal domain with infinitely many maximal ideals $M$ and
  $K$ be its quotient field. Let $I \neq \emptyset$ be a finite set
  and $f_i \in D[x]$ be monic polynomials for $i\in I$ and set $m = \sum_{i \in I}\deg(f_i)$.

  Then for every $n \in \mathbb{N}_0$, there exist monic polynomials $F_i \in D[x]$ for $i\in I$,
  such that
\begin{enumerate}
\item $\deg(F_i) = \deg(f_i)$ for all $i\in I$,
\item the polynomials $F_i$ are irreducible in $K[x]$ and
   pairwise non-associated in $K[x]$, and
\item $F_i \equiv f_i \mod p^{n + 1}D[x]$ for all maximal ideals $(p)$ of $D$
with $\card{\R_M(p)} \leq m$.
\end{enumerate}
\end{lemma}
\begin{remark}\label{Remark:replacements}
We use Lemma~\ref{lemma:replacements} to turn monic polynomials
$f_i \in D[x]$ into monic irreducible polynomials $F_i \in K[x]$ so that,
replacing any $f_i$ by the corresponding $F_i$ preserves the $p$-adic
valuations (for all prime elements $p \in D$) of the values of the product
at each element of $M$. Consequently, the fixed divisor of the original
product coincides with that of the modified product, cf.~\cite[Remark 6.4]{hiebler2023characterizing}.
\end{remark}

\section{Main Results}
In this section, we construct polynomials in $\Int(M, D)$ with prescribed
(multi)sets of lengths and show that $\Int(M, D)$ is not a transfer
Krull domain.

\begin{theorem}\label{Thm:prescribed-lengths}
  Let $D$ be a principal ideal domain such that $\Spec(D)$ is infinite
  and for each $\{0\} \neq M \in \Spec(D)$, the residue
  field $D/M$ is finite. Let $1 \le m_1 \le m_2 \le \cdots \le m_n$ be
  natural numbers.

  Then there exists a polynomial $H \in \Int(M, D)$ with
  exactly $n$ essentially different factorizations
  into irreducible polynomials in $\Int(M, D)$,
  the length of these factorizations being $m_1 + 1$, \ldots, $m_n + 1$.
\end{theorem}

\begin{proof}
We explicitly construct the polynomial $H$ and describe how it factors in $\Int(M, D)$.
Let $M = (p)$, where $p \in D$ is a prime element. Suppose $n = 1$. Then
\[H = \left(\frac{x}{p}\right)^{m_1+1}\]
is an element of $\Int(M, D)$ with exactly one factorization, and this factorization
has length $m_1+1$.

From now, $n > 1$ and by a complete set of residues we
mean a complete set of residues with respect to $M$. We first construct $H \in \Int(M, D)$.

Let $N = \left(\sum_{i=1}^n m_i\right)^2 - \sum_{i=1}^n m_i^2$ and
$\p = \{p_1, p_2, \ldots, p_N\}$ be distinct prime elements of
$D$ with $\card{\R_M(p_i)} > 2$ for each $1 \leq i \leq N$.
Note that by Lemma~\ref{Lemma:FMP}, the set $\p$ exists.
Let $\card{\R_M(p_i)} = t_i$ and without loss of generality,
assume $t_1 \leq t_2 \leq \cdots \leq t_N$.
We denote by $\Q$ the set of prime elements $q$ of $D$ with
$\card{\R_M(q)} \leq t_N$ but $q \not\in \p$.

For each $p_i \in \p$,
we use the Chinese Remainder Theorem to construct a set $T_{p_i}$ which contains exactly
one complete set of residues of $p_i$ but does not
contain a complete set of residues of $q \in \Q$ or $p_j$ for $j \neq i$.
To achieve this, for each $1 \leq i \leq N$, let
$a_{i, 1}, a_{i, 2}, \ldots, a_{i, t_i}$ be a complete set of residues of
$p_i$ such that $a_{i, 1} = 0$ and $a_{i, 2} = p$ for all $i$.

For each $1 \leq i \leq N$, we construct
$T_{p_i} = \{c_{i, 1}, c_{i, 2}, \ldots, c_{i, t_i}\} \subseteq D$
such that
  \begin{enumerate}
  \item For all $1 \leq j \leq {t_i}$, $c_{i, j} \equiv a_{i, j} \Mod{p_i}$
  and $c_{i, j} \not \equiv a_{i, j} \Mod{p_i^2}.$
  \item For all $1 \leq j \leq {t_i}$, $c_{i, j} \not\equiv a_{k, 1} \Mod{p_k}$
  for $i \neq k$.
  \item For all $1 \leq j \leq {t_i}$, $c_{i, j} \not\equiv a_{k, 2} \Mod{p_k}$ for $i \neq k$.
    \item For all $1 \leq i \leq N$ and all $1 \leq j \leq {t_i}$, $c_{i, j} \equiv 1 \Mod{q}$
    for all $q \in \Q$.
  \end{enumerate}
In addition, we find a nonzero element $e \in D$ such that for all $1 \leq i \leq {N}$,
 \begin{enumerate}
  \item $e \equiv a_{i, 1} \Mod{p_i}$ and $e \not \equiv a_{i, 1} \Mod{p_i^2}.$
    \item $e \equiv 1 \Mod{q}$ for all $q \in \Q$.
  \end{enumerate}
Let $w(x)= x - e$ and for each $1 \leq i \leq N$, set
\[g_i = (x - c_{i, 2})(x - c_{i, 3} ) \cdots (x - c_{i, t_i}).\]

For $1 \leq i \leq N$, we use Lemma~\ref{lemma:replacements} to
construct monic polynomials $G_i \in D[x]$ that are irreducible in $K[x]$,
pairwise non-associated in $K[x]$, with $\deg(G_i)= \deg(g_i)$, and such that,
for every selection
of polynomials from among $w(x)$ and $g_i$, the
product of the polynomials has the same relevant properties (with respect to all the
primes $p$ with $\card{\R_M(p)} \leq \deg\left(\prod^{N}_{i = 1}g_i\right))$ as the modified
product in which each $g_i$ has been replaced by
$G_i$ (see Remark~\ref{Remark:replacements}).

We set
\[ \F = \{G_i \mid 1 \leq i \leq N\}\]
and assign indices to the elements of $\F$ as follows
\begin{equation*}
  \F = \{G_{(k,i,h,j)} \mid
  1\le k, h \le n, k\ne h, 1\le i\le m_k, 1\le j\le m_h\},
\end{equation*}
in particular, like the authors in \cite{SF2013:Sets-of-lengths, SF-SN-RR2019:Sets-of-lengths},
we arrange the elements of $\F$ in an $m \times m$ matrix $B$,
where $m = \sum_{i=1}^n m_i$, with the rows and columns of $B$ partitioned into
$n$ blocks of sizes $m_1, \ldots, m_n$.

Now $G_{(k,i,h,j)}$ designates the entry in the $i$-th row of the
$k$-th block of rows and in the $j$-th column of the $h$-th block of
columns. Since no element of $\F$ has row and column index in the
same block, the positions in the blocks of $B$ with block
sizes $m_1,\ldots, m_n$ are left empty, see
\cite[Fig.~1]{SF-SN-RR2019:Sets-of-lengths}.

Following the notation in \cite{SF-SN-RR2019:Sets-of-lengths},
let $I_k = \{(k,i) \mid 1\le i \le m_k\}$ for $1\le k\le n$ and set
\begin{equation*}
  I = \bigcup_{k=1}^n I_k.
\end{equation*}
It follows that
\begin{equation*}
  I = \{(k,i)\mid 1\le k \le n, 1\le i \le m_k\}
\end{equation*} is the set of all possible row indices, or,
equivalently, column indices.

For $(k,i)\in I_k$, let $B[k,i]$ be
the set of all elements $G \in \F$ which are either in row or in
column $(k,i)$ of $B$, see \cite[Fig.~1]{SF-SN-RR2019:Sets-of-lengths}.
In particular,
\begin{equation*}
  B[k,i] = \{G_{(k,i,h,j)} \mid (h,j)\in
  I\setminus I_k\}\cup \{G_{(h,j,k,i)} \mid (h,j)\in I\setminus I_k\}
\end{equation*}

Now, for $(k,i) \in I$, let
\begin{equation*}
 f_i^{(k)}(x) = \prod_{G \in B[k,i]} G.
\end{equation*}

For each $(k,i) \in I$, we again use Lemma~\ref{lemma:replacements} to construct
monic polynomials $F_i^{(k)}(x) \in D[x]$ which are irreducible in $K[x]$, pairwise
non-associated in $K[x]$, with $\deg(F_i^{(k)})= \deg(f_i^{(k)})$,
and such that, for every selection
of polynomials from among $w(x)$ and $f_i^{(k)}$ for $(k,i) \in I$, the
product of the polynomials has the same fixed divisor as the modified
product in which each $f_i^{(k)}$ has been replaced by
$F_i^{(k)}$, cf.~Remark~\ref{Remark:replacements}.

We set
\begin{equation*}
\mathcal{H}(x) = w(x)\! \prod_{(k,i) \in I}\!\! F_i^{(k)}(x)
\quad\textrm{and}\quad
H(x) = \frac{\mathcal{H}(x)}{p_1 \cdots p_N}.
\end{equation*}

Next, we show that $\fixdiv(\mathcal{H}(x)) = p_1 \cdots p_N$; this in particular
implies that $H(x)\in \Int(M, D)$ and $\fixdiv(H(x)) = 1$.
Now, by construction,
\begin{equation}\label{eq:fixdiv=*pi}
 \fixdiv(\mathcal{H}(x)) =
 \fixdiv\left(w(x) \prod_{(k,i) \in I} F_i^{(k)}(x)\right) =
 \fixdiv\left(w(x) \prod_{(k,i)\in I}f_i^{(k)}\right)
\end{equation}
\begin{equation*}
 = \fixdiv\left((x-e)\prod_{1 \leq i\leq N}G_i^2\right) = \fixdiv\left((x-e)\prod_{1 \leq i\leq N}g_i^2\right).
\end{equation*}

Furthermore, for each $1 \leq i \leq N$, the following hold:
\begin{enumerate}
\item $v_{p_i}(G_i(a_{i, 1})) = 0 \text{ and }
v_{p_i}(G_j(a_{i, 1}))) = v_{p_i}(G_j(a_{i, 2})) = 0 \text{ for } j \neq i.$
\item $v_{p_i}(G_i(a_{i, j})) = 1$ for all $2 \leq j \leq t_i$.
\item $v_{p_i}(w(a_{i, 1})) = 1$.
\item 0 is not a root of $(x-e)\prod_{1 \leq i\leq N}G_i^2$
modulo $q$ for all primes $q \in \Q$. More generally, no prime $\mathfrak{p}$ with
$\card{\R_M(\mathfrak{p})} \leq \deg\left(\prod^{N}_{i = 1}G^2_i\right)$ divides
 $\fixdiv(\mathcal{H}(x))$.
\end{enumerate}
It follows that $\fixdiv(\mathcal{H}(x))= p_1 \cdots p_N$.

Lastly, we prove that the essentially different factorization of $H(x)$
into irreducibles in $\Int(M, D)$ are given by
\begin{equation}\label{eq:facts-of-H}
  H(x) =  F_1^{(h)}(x)\cdots F_{m_h}^{(h)}(x) \cdot \frac{w(x)\prod_{(k,i)\in
      I\setminus I_h} F_i^{(k)}(x)}{p_1 \cdots p_N}
\end{equation}
where $1 \le h \le n$.

Now since by the construction of $g_i$ and $G_i$,
$$\fixdiv\left((x-e)\prod_{1 \leq i\leq N}g_i^2\right) = \fixdiv\left((x-e)\prod_{1 \leq i\leq N}g_i\right) = \fixdiv\left((x-e)\prod_{1 \leq i\leq N}G_i\right),$$ it follows
from Equation \eqref{eq:fixdiv=*pi} that
\begin{equation*}
  \fixdiv\left(w(x) \prod_{(k,i) \in I} F_i^{(k)}(x)\right) =
   \fixdiv\left((x-e)\prod_{1 \leq i\leq N}G_i\right) = \fixdiv\left(w(x)\prod_{(k,i)\in
      I\setminus I_h} f_i^{(k)}(x)\right)
\end{equation*}
for some $1 \le h \le n$. Hence by Lemma~\ref{lemma:replacements}, we have
\begin{equation*}
  \fixdiv\left(w(x) \prod_{(k,i) \in I} F_i^{(k)}(x)\right) = \fixdiv\left(w(x)\prod_{(k,i)\in
      I\setminus I_h} F_i^{(k)}(x)\right)
\end{equation*}
for some $1 \le h \le n$. Furthermore, the polynomial $\dfrac{w(x)\prod_{(k,i)\in
      I\setminus I_h} F_i^{(k)}(x)}{p_1 \cdots p_N}$ is irreducible in $\Int(M, D)$, by Theorem~\ref{Thm-irr}.
Therefore the factorizations in Equation~\eqref{eq:facts-of-H} are the
$n$ essentially different factorizations of $H(x)$.
\end{proof}

In the next theorem, we fix an irreducible element $\frac{x}{p}$
and construct another irreducible element $H \in \Int(M, D)$
such that the lengths of factorizations of $\frac{x}{p}\cdot H$ are not bounded.
This implies by Remark~\ref{Remark:notTK} that there does not exist a transfer
homomorphism from the multiplicative monoid  $\Int(M, D) \setminus \{0\}$ to a
block monoid. In particular, $\Int(M, D)$ is not a transfer Krull Domain.

\begin{theorem}\label{Thm:non-transfer-krull}
  Let $D$ be a principal ideal domain such that $\Spec(D)$ is infinite
  and for each $\{0\} \neq M \in \Spec(D)$, the residue
  field $D/M$ is finite.
  Then for every $n \geq 1$ there exist irreducible elements
  $H, G_1, \ldots, G_{n + 1}$ in $\Int(M, D)$ such that
  \[\frac{x}{p}H = G_1 \cdots G_{n+1},\]
  where $p$ is the generator of $M$.
\end{theorem}

\begin{proof}
By a complete set of residues we mean a complete set of residues with respect to $M$.
Now let $\p = \{p_1, p_2, \ldots, p_n\}$ be distinct prime elements of
$D$ with $\card{\R_M(p_i)} > 1$ for each $1 \leq i \leq n$. Let
$\card{\R_M(p_i)} = t_i$ and without loss of generality,
assume $t_1 \leq t_2 \leq \cdots \leq t_n$.
Let $\Q$ be the set of prime elements $q$ of $D$ with
$\card{\R_M(q)} \leq t_n + n$ but $q \not\in \p$.

We use the Chinese Remainder Theorem to construct a set $T$ which contains exactly
one complete set of residues of $p_i$ for all $1 \leq i \leq n$, but does not
contain a complete set of residues of any $q \in \Q$.
To achieve this, for each $1 \leq i \leq n$, let
$a_{i, 1}, a_{i, 2}, \ldots, a_{i, t_i}$ be a complete set of residues of
$p_i$ such that $a_{i, 1} \equiv 0 \Mod{p_i}$ for all $i$.

Let $T = \{r_1, r_2, \ldots, r_{t_n}\} \subset D$ with the following properties:
  \begin{enumerate}
  \item For all $1 \leq i \leq n$, $r_1 \equiv 0 \Mod{p_i}$.
  \item For all $1 \leq i \leq n$ and all $2 \leq j \leq {t_i}$, $r_j \equiv a_{i, j} \Mod{p_i}$
  and $r_j \not \equiv a_{i, j} \Mod{p_i^2}.$
  \item For all $1 \leq i \leq n$ and all $j > {t_i}$, $r_j \equiv p \Mod{p_i}$.
    \item For all $1 \leq j \leq n$, $r_j \equiv 1 \Mod{q}$
    for all $q \in \Q$.
  \end{enumerate}
Let $g = (x - r_2)(x - r_3) \cdots (x - r_n)$. Using Lemma \ref{lemma:replacements},
we construct a monic polynomial $G \in D[x]$
that is irreducible in $K[x]$, with $\deg(G)= \deg(g)$, and such that,
for every selection of polynomials from among $x$, $g$, and $x - a_{i, 1}$ for $1 \leq i \leq n$, the product of the polynomials has the same fixed divisor as the modified product in
which $g$ has been replaced by $G$ (see Remark~\ref{Remark:replacements}). Let
\[H = \frac{G\cdot(x -  a_{1, 1})(x -  a_{2, 1}) \cdots (x -  a_{n, 1})}{p_1p_2 \cdots p_n}.\]
It follows by Theorem~\ref{Thm-irr} that $H$ is irreducible in $\Int(M, D)$. Furthermore,
\[\frac{x}{p}H = \frac{xG}{pp_1p_2 \cdots p_n} \cdot (x -  a_{1, 1})(x -  a_{2, 1}) \cdots (x -  a_{n, 1})\]
and $\frac{xG}{pp_1p_2 \cdots p_n}$ and each $x -  a_{i, 1}$ are irreducible in $\Int(M, D)$. Therefore,
the assertion follows.
\end{proof}

\begin{corollary}
Let $D$ be a principal ideal domain such that $\Spec(D)$ is infinite
and for each $\{0\} \neq M \in \Spec(D)$, the residue field $D/M$ is finite.
Then $\Int(M, D)$ is not a transfer Krull domain.
\end{corollary}

\section*{Acknowledgement}
This research was initiated as part of the Inaugural African Women in Algebra
Workshop, which took place at Kabale University, Uganda, in 2024, organized by
the African Women in Algebra network. The authors gratefully acknowledge Kabale
University for its hospitality and thank Prof.~Sophie Frisch and the following
organizations for their generous financial support, which made the workshop possible:
the African Women in Mathematics Association, the
Centre International de Math\'{e}matiques Pures et Appliqu\'{e}es, the European Mathematical Society-Committee
for Developing Countries, the International Mathematical Union-Commission for Developing
Countries, the International Mathematical Union-Committee for Women in
Mathematics, the International Science Programme, Kabale University, the London
Mathematical Society, and the University of Vienna.

The authors are also grateful to Dr.~Roswitha Rissner for her helpful
comments on the exposition of the manuscript.
\bibliographystyle{plain}
\bibliography{bibliographyAWA}

\end{document}